\documentclass[10pt]{amsart}
\usepackage[T1]{fontenc}
\usepackage[latin5]{inputenc}
\usepackage{float}
\usepackage{amssymb}
\usepackage{amsfonts}
\usepackage[centertags]{amsmath}
\usepackage{amsthm}
\usepackage{newlfont}
 \makeatletter
\theoremstyle{plain}
\newtheorem{thm}{Theorem}[section]
\numberwithin{equation}{section}
\numberwithin{figure}{section}  
\theoremstyle{plain}

\theoremstyle{plain}

\theoremstyle{plain}
\newtheorem{cor}[thm]{Corollary} 
\theoremstyle{plain}

\theoremstyle{plain}
\newtheorem{lem}[thm]{Lemma} 
\theoremstyle{plain}

\begin{document}
\title{The Sequence Space $bv$ and Some Applications}
\author{Murat Kiri\c{s}ci}
\address{Department of Mathematical Education, Hasan Ali Y\"{u}cel Education Faculty, Istanbul University, Vefa, 34470, Fatih, Istanbul, Turkey}
\email{mkirisci@hotmail.com, murat.kirisci@istanbul.edu.tr}
\vspace{0.5cm}
\vspace{0.5cm}

\begin{abstract}
In this work, we give well-known results related to some properties, dual spaces and matrix transformations of the sequence space $bv$ and introduce the matrix domain of space $bv$ with arbitrary triangle matrix $A$. Afterward, we choose the matrix $A$ as Ces\`{a}ro mean of order one, generalized weighted mean and Riesz mean and compute $\alpha-, \beta-, \gamma-$duals of these spaces. And also, we characterize the matrix classes of the new spaces.
\end{abstract}

\maketitle

\section{Introduction}
The set of all sequences denotes with $\omega := \mathbb{C}^{\mathbb{N}}:=\{x=(x_{k}): x: \mathbb{N}\rightarrow \mathbb{C}, k\rightarrow x_{k}:=x(k)\}$, where $\mathbb{C}$ denotes the complex field and $\mathbb{N}=\{0,1,2,\ldots\}$. Each linear subspace of $\omega$ (with the induced addition and scalar multiplication) is called a \emph{sequence space}. We will write $\phi, \ell_{\infty}, c ~\textrm{ and }~ c_{0}$ for the sets of all finite, bounded, convergent and null sequences, respectively. It obviously that these sets are subsets of $\omega$.\\

A sequence, whose $k-th$ term is $x_{k}$, is denoted by $x$ or $(x_{k})$. By $e$ and $e^{(n)}, (n=0,1,2,...)$, we denote the sequences such that $e_{k}=1$ for $k=0,1,2,...$, and $e_{n}^{(n)}=1$ and $e_{k}^{(n)}=0$ for $k\neq n$.\\

\emph{A coordinate space} (or \emph{$K-$space}) is a vector space of numerical sequences, where addition and scalar multiplication are defined pointwise. That is, a sequence space $X$ with a linear topology is called a $K$-space provided each of the maps $p_{i}:X\rightarrow \mathbb{C}$ defined by $p_{i}(x)=x_{i}$ is continuous for all $i\in \mathbb{N}$. A $BK-$space is a $K-$space, which is also a Banach space with continuous coordinate functionals $f_{k}(x)=x_{k}$, $(k=1,2,...)$. A $K-$space $K$ is called an \emph{$FK-$space} provided $X$ is a complete linear metric space. An \emph{$FK-$space} whose topology is normable is called a \emph{$BK-$ space}. A sequence $(b_{n}), (n=0,1,2,\ldots)$ in a linear metric $X$ is called a \emph{Schauder basis} if for each $x\in X$ there exists a unique sequence $(\alpha_{n}), (n=0,1,2,\ldots)$ of scalars such that $x=\sum_{n=0}^{\infty}\alpha_{n}b_{n}$. An \emph{$FK-$space} $X$ is said to have $AK$ property, if $\phi \subset X$ and $\{e^{(n)}\}$ is a basis for $X$ and $\phi=span\{e^{(n)}\}$, the set of all finitely non-zero sequences.\\

The series $\sum\alpha_{k}b_{k}$ which has the sum $x$ is then called the expansion of $x$ with respect to $(b_{n})$, and written as $x=\sum\alpha_{k}b_{k}$. An \emph{$FK-$space} $X$ is said to have $AK$ property, if $\phi \subset X$ and $\{e^{k}\}$ is a basis for $X$, where $e^{k}$ is a sequence whose only non-zero term is a $1$ in $k^{th}$ place for each $k\in \mathbb{N}$ and $\phi=span\{e^{k}\}$, the set of all finitely non-zero sequences.\\

Let $X$ is a sequence space and $A$ is an infinite matrix. The sequence space
\begin{eqnarray}\label{eq0}
X_{A}=\{x=(x_{k})\in\omega:Ax\in X\}
\end{eqnarray}
is called the matrix domain of $X$ which is a sequence space(for several examples of matrix domains, see \cite{Basarkitap} p. 49-176).\\

We write $\mathcal{U}$ for the set of all sequences $u=(u_{k})$ such that $u_{k}\neq 0$ for all $k\in \mathbb{N}$. For $u\in \mathcal{U}$, let $1/u=(1/u_{k})$. Let $u,v\in \mathcal{U}$, $(t_{k})$ be a sequence of positive and write $T_{n}=\sum_{k=0}^{n}t_{k}$. Now, we define the difference matrix $\Delta=(\delta_{nk})$, the matrix $C=(c_{nk})$ of the Cesaro mean of order one, the \emph{generalized weighted mean} or \emph{factorable matrix} $G(u,v)=(g_{nk})$ and the matrix $R^{t}=\left(r_{nk}^{t}\right)$ of the Riesz mean by
\begin{eqnarray}\label{deltamtrx}
\delta_{nk}= \left\{ \begin{array}{ccl}
(-1)^{n-k}&, & \quad (n-1\leq k \leq n)\\
0&, & \quad (0\leq k < n-1 ~\textrm{ or }~ k>n)
\end{array}\right.
\end{eqnarray}

\begin{eqnarray*}
c_{nk}=\left\{\begin{array}{ccl}
\frac{1}{n+1}&, & (0\leq k\leq n)\\
0&, & (k>n)
\end{array}\right.
\end{eqnarray*}

\begin{eqnarray*}
g_{nk}=\left\{\begin{array}{ccl}
u_{n}v_{k}&, & (0\leq k\leq n)\\
0&, & (k>n)
\end{array}\right.
\end{eqnarray*}

\begin{eqnarray*}
r_{nk}^{t}=\left\{\begin{array}{ccl}
t_{k}/T_{n}&, & (0\leq k\leq n)\\
0&, & (k>n)
\end{array}\right.
\end{eqnarray*}
for all $k,n\in\mathbb{N}$; where $u_{n}$ depends only on $n$ and $v_{k}$ only on $k$.\\

In this work, we give well-known results related to some properties, dual spaces and matrix transformations of the sequence space $bv$ and introduce the matrix domain of space $bv$ with arbitrary triangle matrix $A$. Afterward, we choose the matrix $A$ as Ces\`{a}ro mean of order one, generalized weighted mean and Riesz mean and compute $\alpha-,\beta-\beta\gamma-$duals of these spaces. And also, we characterize the matrix classes of the spaces $bv(C)$, $bv(G)$, $bv(R)$.

\section{Well-Known Results}

In this section, we will give some well-known results and will define a new form of the sequence space $bv$ with arbitrary triangle $A$.

The space of all sequences of bounded variation defined by
\begin{eqnarray*}
bv=\left\{x=(x_{k})\in \omega: \sum_{k=1}^{\infty}|x_{k}-x_{k-1}|<\infty\right\},
\end{eqnarray*}
which is a $BK-$space under the norm $\|x\|_{bv}=|x_{0}|+\sum_{k=1}^{\infty}|x_{k}-x_{k-1}|$
for $x\in bv$. The space $bv_{0}$ denotes $bv_{0}=bv\cap c_{0}$. It is clear that $bv=bv_{0}+\{e\}$.
Also the inclusions $\ell_{1}\subset bv_{0} \subset bv=bv_{0}+\{e\}\subset c$ are strict.\\

Let $X$ be a sequence space and $\Delta$ denotes the matrix as defined by (\ref{deltamtrx}).
The matrix domain $X_{\Delta}$ for $X=\{\ell_{\infty}, c, c_{0}\}$ is called the \emph{difference sequence spaces},
which was firstly defined and studied by Kizmaz \cite{Kizmaz}. If we choose $X=\ell_{1}$, the space $\ell_{1}({\Delta})$
is called \emph{the space of all sequences of bounded variation} and denote by $bv$. In \cite{AB2}, Ba\c{s}ar and Altay
have defined the sequence space $bv_{p}$ for $1\leq p \leq \infty$ which consists of all sequences such that
$\Delta$-transforms of them are in $\ell_{p}$ and have studied several properties of these spaces. Basar and Altay\cite{AB2}
proved that the space $bv_{p}$ is a $BK-$ space with the norm $\|x\|_{bv_{p}}=\|\Delta x\|_{\ell_{p}}$
and linearly isomorphic to $\ell_{p}$ for $1\leq p \leq \infty$. The inclusion relations for the space $bv_{p}$
are given in \cite{AB2} as below:
\begin{itemize}
\item[(i)] The inclusion $\ell_{p}\subset bv_{p}$ strictly holds for $1\leq p \leq \infty$.
\item[(ii)] Neither of the spaces $bv_{p}$ and $\ell_{\infty}$ includes the other one, where $1<p<\infty$.
\item[(iii)] If $1\leq p < s$, then $bv_{p}\subset bv_{s}$.
\end{itemize}

Define a sequence $b^{(k)}=\{b_{n}^{(k)}\}_{n\in \mathbb{N}}$ of elements of the space $bv_{p}$ for every fixed $k\in \mathbb{N}$
by $0$ for $n<k$ and $1$ for $n\geq k$. Then the sequence $\{b_{n}^{(k)}\}_{n\in \mathbb{N}}$ is a Schauder basis for the space $bv_{p}$
and every sequence $x\in bv_{p}$ has a unique representation $x=\sum_{k}(\Delta x)_{k}b^{(k)}$ for all $k\in \mathbb{N}$. The space $bv_{\infty}$
has no Schauder basis\cite{AB2}.

Let $X$ and $Y$ be two sequence spaces, and $A=(a_{nk})$ be an infinite matrix of complex numbers $a_{nk}$,
where $k,n\in\mathbb{N}$. Then, we say that $A$ defines a \emph{matrix mapping} from $X$ into $Y$, and we denote
it by writing $A :X \rightarrow Y$ if for every sequence $x=(x_{k})\in X$. The sequence $Ax=\{(Ax)_{n}\}$,
the $A$-transform of $x$, is in $Y$; where
\begin{eqnarray}\label{triA}
(Ax)_{n}=\sum_{k}a_{nk}x_{k}~\textrm{ for each }~n\in\mathbb{N}.
\end{eqnarray}
For simplicity in notation, here and in what follows, the summation without limits runs from $0$ to $\infty$.
By $(X:Y)$, we denote the class of all matrices $A$ such that $A:X \rightarrow Y $. Thus, $A\in(X:Y)$ if and
only if the series on the right side of (\ref{triA}) converges for each $n\in\mathbb{N}$ and each $x\in X$ and
we have $Ax =\{(Ax)_{n}\}_{n \in\mathbb{N}}\in Y$ for all $x\in X$. A sequence $x$ is said to be \emph{$A$-summable
to $l$} if $Ax$ converges to $l$ which is called the \emph{$A$-limit} of $x$.\\

If $X,Y\subset \omega$ and $z$ any sequence, we can write $z^{-1}*X=\{x=(x_{k})\in \omega: xz\in X\}$ and $M(X,Y)=\bigcap_{x\in X}x^{-1}*Y$.
If we choose $Y=\ell_{1},cs, bs$, then we obtain the $\alpha-, \beta-, \gamma- $duals of $X$, respectively as

\begin{eqnarray*}
X^{\alpha}&=&M(X,\ell_{1})=\{a=(a_{k})\in \omega:  ax=(a_{k}x_{k})\in \ell_{1} ~\textrm{for all }~   x\in X\}\\
X^{\beta}&=&M(X,cs)=\{a=(a_{k})\in \omega:  ax=(a_{k}x_{k})\in cs ~\textrm{for all }~   x\in X\}\\
X^{\gamma}&=&M(X,bs)=\{a=(a_{k})\in \omega:  ax=(a_{k}x_{k})\in bs ~\textrm{for all }~   x\in X\}.
\end{eqnarray*}
The $\alpha-, \beta-, \gamma- $duals of classical sequence spaces are defined by
\begin{eqnarray*}
c_{0}^{\alpha}=c^{\alpha}=\ell_{\infty}^{\alpha}=\ell_{1}, \quad \quad \ell_{1}^{\alpha}=\ell_{\infty}, \quad \quad cs^{\alpha}=bs^{\alpha}=bv^{\alpha}=bv_{0}^{\alpha}=\ell_{1}\\
c_{0}^{\beta}=c^{\beta}=\ell_{\infty}^{\beta}=\ell_{1}, \quad \ell_{1}^{\beta}=\ell_{\infty}, \quad cs^{\beta}=bv, \quad bs^{\beta}=bv_{0}, \quad bv^{\beta}=cs, \quad bv_{0}^{\beta}=bs\\
c_{0}^{\gamma}=c^{\gamma}=\ell_{\infty}^{\gamma}=\ell_{1}, \quad \quad \ell_{1}^{\gamma}=\ell_{\infty},  \quad \quad cs^{\gamma}=bs^{\gamma}=bv, \quad bv^{\gamma}=bv_{0}^{\gamma}=bs.
\end{eqnarray*}

The \emph{continuous dual} of $X$ is the space of all continuous linear functionals on $X$ and is denote by $X^{*}$.
The continuous dual $X^{*}$ with the norm $\|.\|^{*}$ defined by $\|f\|^{*}=\sup\left\{|f(x)|: \|x\|=1\right\}$ for $f\in X^{*}$.\\

The spaces $c^{*}$ and $c_{0}^{*}$ are norm isomorphic with $\ell_{1}$. Also, $\|a\|_{\beta}=\|a\|_{\ell_{1}}$ for all $a\in \ell_{\infty}^{\beta}$.\\

The following theorem is the most important result in the theory of matrix transformations:

\begin{thm}
Matrix transformations between $BK-$spaces are continuous.
\end{thm}

In \cite{AB}, Ba\c{s}ar and Altay developed very useful tools for duals and matrix transformations of sequence spaces as below:

\begin{thm}\cite[Lemma 5.3]{AB2}
Let $X, Y$ be any two sequence spaces, $A$ be an infinite matrix and $U$ a triangle matrix matrix.Then, $A\in (X: Y_{U})$ if and only if $UA\in (X:Y)$.
\end{thm}

\begin{thm}\cite[Theorem 3.1]{AB}\label{mtrxtool}
$B_{\mu}^{U}=(b_{nk})$ be defined via a sequence $a=(a_{k})\in\omega$ and inverse of the triangle matrix $U=(u_{nk})$ by
\begin{eqnarray*}
b_{nk}=\sum_{j=k}^na_{j}v_{jk}
\end{eqnarray*}
for all $k,n\in\mathbb{N}$. Then,
\begin{eqnarray*}
\lambda_{U}^{\beta}=\{a=(a_{k})\in\omega: B^{U}\in(\lambda:c)\}
\end{eqnarray*}
and
\begin{eqnarray*}
\lambda_{U}^{\gamma}=\{a=(a_{k})\in\omega: B^{U}\in(\lambda:\ell_{\infty})\}.
\end{eqnarray*}
\end{thm}

The following theorem proved by Zeller \cite{Zeller}:
\begin{thm}
Let $X$ be an $FK-$space whose topology is given by means of the seminorms $\{q_{n}\}_{n=1}^{\infty}$
and let $A$ be an infinite matrix. Then $X_{A}$ is an $FK-$space when topologized by
\begin{eqnarray*}
&x& \rightarrow |x_{j}| \quad\quad (j=1,2,\ldots)\\
&x& \rightarrow \sup_{n}|\sum_{j=1}^{n}a_{ij}x_{j}| \quad\quad (i=1,2,\ldots)\\
&x& \rightarrow q_{n}(Ax) \quad\quad (n=1,2,\ldots).
\end{eqnarray*}
\end{thm}
Following theorems are given by Bennet \cite{Bennet}.
\begin{thm}
An $FK-$space $X$ contains $\ell_{1}$ if and only if $\{e^{(j)}: j=1,2,\ldots\}$ is a bounded
subsets of $X$.
\end{thm}

\begin{thm}
Let $A$ be a matrix and $X$ an $FK-$sace. Then $A$ maps $\ell_{1}$ into $X$ if and only if the columns of $A$
 belong to $X$ and form a bounded subset there.
\end{thm}

\begin{thm}\label{thmbennet1}
An $FK-$space contains $bv_{0}(bv)$ if and only if $(e\in E)$ and $\{\sum_{j=1}^{n}e^{(j)}: j=1,2,\ldots\}$
is a bounded subsets of $X$.
\end{thm}

\begin{thm}
Let $A$ be a matrix and $X$ an $FK-$sace. Then $A$ maps $bv_{0}$ into $X$ if and only if the columns of $A$
 belong to $X$ and their partial sums form a bounded subset there.
\end{thm}

The important class of co-null $FK-$space was introduced by Snyder \cite{Snyder}: An $FK-$space $X$ containing $\varphi \cup \{e\}$
is said to be \emph{co-null} if $\sum_{j=1}^{\infty}e^{j}=e$ weakly in $X$.\\

Therefore, using the definition of co-null and Theorem \ref{thmbennet1}, we have:

\begin{cor}\cite{Bennet}
Any co-null $FK-$space must contain $bv$.
\end{cor}

\begin{thm}[Theorem 4.3.2 in \cite{Wil84}]
Let $(X, \|.\|)$ be a $BK-$space. Then $X^{\beta}$ is a $BK-$space with $\|a\|_{\beta}=\sup\left\{\sup_{n}|\sum_{k=0}^{n}a_{k}x_{k}|: \|x\|=1\right\}$.
\end{thm}

\begin{thm}[Theorem 7.2.9 in \cite{Wil84}]
The inclusion $X^{\beta} \subset X^{*}$ holds in the following sense: Let the  $ \widehat{}:X^{\beta} \rightarrow X^{*}$ be defined by
 $\widehat{}(a)=\widehat{a}: X \rightarrow \mathbb{C}, (a\in X^{\beta})$ where $\widehat{a}(x)=\sum_{k=0}^{\infty}a_{k}x_{k}$ for all
 $x\in X$. Then $\widehat{}$ is an isomorphism into $X^{*}$. If $X$ has $AK$, then the map $\widehat{}$ is onto $X^{*}$.
\end{thm}

Now, we define the space $bv$ with a lower triangle matrix $A=(a_{nk})$ for all $k,n\in \mathbb{N}$ as below:
\begin{eqnarray}\label{trimtrx}
bv(A)=\left\{x=(x_{k})\in \omega: Ax \in bv\right\}.
\end{eqnarray}

The space $bv$ is $BK-$space with the norm $\|x\|_{bv}=\|\Delta x\|_{\ell_{1}}$ and $A$ is a triangle matrix. Then, from Theorem 4.3.2 in \cite{Wil84},
we say that the space $bv(A)$ is a $BK-$space with the norm $\|x\|_{bv(A)}=\|Ax\|_{bv}$.\\

If $A$ is a triangle, then one can easily observe that the sequence $X_{A}$ and $X$ are linearly isomorphic,i.e., $X_{A}\cong X$.
Therefore, the spaces $bv$ and $bv(A)$are norm isomorphic to the spaced $\ell_{1}$ and $bv$, respectively.

Now we give some results:
\begin{lem}\label{duallem1}
Let $A=(a_{nk})$ be an infinite matrix.  Then, the following statements hold:
\begin{itemize}
\item[(i)] $A\in (\ell_{1}:\ell_{\infty})$ if and only if
\begin{eqnarray} \label{deq1}
\sup_{k,n \in \mathbb{N}}|a_{nk}| < \infty.
\end{eqnarray}

\item[(ii)] $A\in (\ell_{1}:c)$ if and only if (\ref{deq1}) holds, and there are $\alpha_{k},\in \mathbb{C}$ such that
\begin{eqnarray}\label{deq2}
\lim_{n \rightarrow \infty} a_{nk} = \alpha_{k}~ \textrm{ for each }~k\in\mathbb{N}.
\end{eqnarray}

\item[(iii)] $A\in (\ell_{1}:\ell_{1})$ if and only if
\begin{eqnarray}\label{deq3}
\sup_{k \in \mathbb{N}}\sum_{n}|a_{nk}| < \infty.
\end{eqnarray}
\end{itemize}
\end{lem}

\section{Applications}

In this section, we give some results related to the space $bv(A)$ as choose well-known matrices instead of arbitrary matrix $A$.\\

\subsection{The space $bv(C)$}

We give the matrix $C_{1}=(c_{nk})$ instead of the matrix $A=(a_{nk})$, in (\ref{trimtrx}).
Then, we obtain the space $bv(C)$ as below:
\begin{eqnarray}\label{trimtrx0}
bv(C)=\left\{x=(x_{k})\in \omega: \sum_{k}\left|\frac{1}{k+1}\sum_{j=0}^{k}x_{j}-\frac{1}{k}\sum_{j=0}^{k-1}x_{j}\right|<\infty\right\}.
\end{eqnarray}
Using the notation (\ref{eq0}), we can denotes the space $bv(C)$ as $bv(C)=(bv)_{C_{1}}=(\ell_{1})_{\Delta.C_{1}}$, where $\Phi=\phi_{nk}=\Delta.C_{1}$ defined by $\phi_{nk}=\frac{-1}{(n+1).n} \quad (0\leq k<n)$; $\phi_{nk}=\frac{1}{n} \quad (k=n)$ and $\phi_{nk}=0 \quad (k>n)$ for all $n$.\\

The $\Phi-$transform of the sequence $x=(x_{k})$ defined by
\begin{eqnarray}\label{seqtrns}
y_{k}=\frac{x_{k}}{k+1}-\frac{1}{k(k+1)}\sum_{j=0}^{k}x_{j}.
\end{eqnarray}

\begin{thm}
Define a sequence $t^{(k)}=\{t_{n}^{(k)}\}_{n\in\mathbb{N}}$ of elements of the space $bv(C)$ for every fixed $k\in \mathbb{N}$ by
\begin{eqnarray*}
t_{n}^{(k)}= \left\{ \begin{array}{ccl}
(k+1)&, & \quad (n=k)\\
1&, & \quad (n<k)\\
0&, & \quad (n>k)
\end{array} \right.
\end{eqnarray*}
Therefore, the sequence $\{t^{(k)}\}_{k\in\mathbb{N}}$ is a basis for the space $bv(C)$ and any $x\in bv(C)$ has a unique representation of the form
\begin{eqnarray*}
x=\sum_{k}(\Phi x)_{k}t^{(k)}.
\end{eqnarray*}
\end{thm}

This theorem can be proved as Theorem 3.1 in \cite{AB2}, we omit details.\\

\begin{thm}
The $\alpha-$dual of the space $bv(C)$ is the set
\begin{eqnarray*}
d_{1}=\left\{a=(a_{k})\in \omega: \sup_{k\in \mathbb{N}}\sum_{n}\left|b_{nk}\right|<\infty \right\}
\end{eqnarray*}
where the matrix $B=(b_{nk})$ is defined via the sequence $a=(a_{n})\in\omega$ by $(n+1)a_{n} \quad (n=k)$, $a_{k} \quad (k<n)$
and $0 \quad (k>n)$.
\end{thm}
\begin{proof}
Let $a=(a_{n})\in \omega$. Using the relation (\ref{seqtrns}), we obtain
\begin{eqnarray}\label{alphaequ}
a_{n}x_{n}=\sum_{n}\left|\sum_{k=0}^{n-1}y_{k}+(n+1)y_{n}\right|a_{n}=(By)_{n}
\end{eqnarray}
It follows from (\ref{alphaequ}) that $ax=(a_{n}x_{n})\in\ell_{1}$ whenever $x\in bv(C)$ if and only if $By\in \ell_{1}$ whenever $y\in \ell_{1}$. We obtain that $a\in\left[bv(C)\right]^{\alpha}$ whenever $x\in bv(C)$ if and only if $B\in (\ell_{1}:\ell_{1})$. Therefore, we get by Lemma \ref{duallem1} (iii) with $B$ instead of $A$ that $a\in\left[bv(C)\right]^{\alpha}$ if and only if $\sup_{k\in \mathbb{N}}\sum_{n}\left|b_{nk}\right|<\infty$.
This gives us the result that $\left[bv(C)\right]^{\alpha}=d_{1}$.
\end{proof}

\begin{thm}
Define the sets by
\begin{eqnarray*}
d_{2}&=&\left\{a=(a_{k})\in \omega: \lim_{n}d_{nk} ~ \textrm{ exists for each }~k\in\mathbb{N}\right\}\\
d_{3}&=&\left\{a=(a_{k})\in \omega: \sup_{k}\sum_{n}\left|d_{nk}\right|<\infty \right\}
\end{eqnarray*}
where the matrix $D=(d_{nk})$ is defined via the sequence $a=(a_{n})\in\omega$ by $(n+1)a_{n} \quad (n=k)$, $\sum_{j=k}^{n}a_{j} \quad (k>n)$
and $0 \quad (k<n)$. Then, $\left[bv(C)\right]^{\beta}=d_{2}\cap d_{3}$.
\end{thm}
\begin{proof}
Consider the equation
\begin{eqnarray}\label{betaequ}
\sum_{k=0}^{n}a_{k}x_{k}=\sum_{k=0}^{n}a_{k}\left[\sum_{j=0}^{k-1}y_{j}+(k+1)y_{k}\right]=(Dy)_{n}  \quad \quad (n\in \mathbb{N}).
\end{eqnarray}
Therefore, we deduce from Lemma \ref{duallem1} (ii) with (\ref{betaequ}) that $ax=(a_{n}x_{n})\in cs$ whenever $x\in bv(C)$ if and only if $Dy\in c$ whenever $y\in \ell_{1}$. From (\ref{deq1}) and (\ref{deq2}), we have
\begin{eqnarray*}
\lim_{n}d_{nk}=\alpha_{k}  ~ \textrm{ and }~ \sup_{k}\sum_{n}\left|d_{nk}\right|<\infty
\end{eqnarray*}
which shows that $\left[bv(C)\right]^{\beta}=d_{2}\cap d_{3}$.
\end{proof}

\begin{thm}
$\left[bv(C)\right]^{\gamma}=d_{3}$.
\end{thm}

\begin{proof}
We obtain from Lemma \ref{duallem1} (i) with (\ref{betaequ}) that $ax=(a_{n}x_{n})\in bs$ whenever $x\in bv(C)$ if and only if $Dy\in \ell_{\infty}$ whenever $y\in \ell_{1}$. Then, we see from (\ref{deq1}) that $\left[bv(C)\right]^{\gamma}=d_{3}$.\\
\end{proof}

In this subsection to use, we define the matrices for brevity that
\begin{eqnarray*}
\widetilde{a}_{nk}=\sum_{j=0}^{k-1}a_{nj}+(k+1)a_{nk} \quad ~ \textrm{ and }~ \quad \widetilde{b}_{nk}=\frac{a_{nk}}{n+1}-\frac{1}{n(n+1)}\sum_{j=0}^{n}a_{jk}
\end{eqnarray*}
for all $k,n \in \mathbb{N}$.

\begin{thm}
Suppose that the entries of the infinite matrices $A=(a_{nk})$ and $E=(e_{nk})$ are connected with the relation
\begin{eqnarray}\label{mtrtrf1}
e_{nk}=\widetilde{a}_{nk}
\end{eqnarray}
for all $k,n\in \mathbb{N}$ and $Y$ be any given sequence space. Then, $A\in (bv(C):Y)$ if and only if $\{a_{nk}\}_{k\in \mathbb{N}}\in [bv(C)]^{\beta}$ for all $n\in \mathbb{N}$ and $E\in (\ell_{1}:Y)$.
\end{thm}
\begin{proof}

Let $Y$ be any given sequence space. Suppose that (\ref{mtrtrf1}) holds between $A=(a_{nk})$ and $E=(e_{nk})$, and take into account that the spaces $bv(C)$ and $\ell_{1}$ are norm isomorphic.\\

Let $A\in (bv(C):Y)$ and take any $y=(y_{k})\in \ell_{1}$. Then $\Phi E$ exists and $\{a_{nk}\}_{k\in\mathbb{N}}\in \{bv(C)\}^{\beta}$ which yields that (\ref{mtrtrf1}) is necessary and $\{e_{nk}\}_{k\in\mathbb{N}}\in (\ell_{1})^{\beta}$ for each $n\in \mathbb{N}$. Hence, $Ey$ exists for each $y\in \ell_{1}$ and thus \begin{eqnarray*}
\sum_{k}e_{nk}y_{k}=\sum_{k}a_{nk}x_{k}~ \textrm{ for all }~ n\in\mathbb{N},
\end{eqnarray*}
we obtain that $Ey=Ax$ which leads us to the consequence $E\in (\ell_{1}:Y)$.\\

Conversely, let $\{a_{nk}\}_{k\in\mathbb{N}}\in \{bv(C)\}^{\beta}$ for each $n\in \mathbb{N}$ and $E\in (\ell_{1}:Y)$ hold, and take any $x=(x_{k})\in bv(C)$. Then, $Ax$ exists. Therefore, we obtain from the equality
\begin{eqnarray*}
\sum_{k=0}^{m}a_{nk}x_{k}=\sum_{k=0}^{m}a_{nk}\left[\sum_{j=0}^{k-1}y_{j}+(k+1)y_{k}\right]=\sum_{k=0}^{m}\left(\sum_{j=k}^{m-1}a_{nj}y_{k}\right)+(m+1)a_{nm}y_{m}   \quad ~ \textrm{ for all }~m,n\in \mathbb{N}
\end{eqnarray*}
as $m\rightarrow\infty$ that $Ax=Ey$ and this shows that $E\in(\ell_{1}:Y)$. This completes the proof.\\
\end{proof}

\begin{thm}
Suppose that the entries of the infinite matrices $B=(\widetilde{b}_{nk})$ and $F=(f_{nk})$ are connected with the relation
\begin{eqnarray}\label{mtrtrf2}
f_{nk}=\widetilde{b}_{nk}
\end{eqnarray}
for all $k,n\in \mathbb{N}$ and $Y$ be any given sequence space. Then, $B\in (Y:bv(C))$ if and only if $F\in (Y:\ell_{1})$.\\
\end{thm}

\begin{proof}
Let $z=(z_{k})\in Y$ and consider the following equality
\begin{eqnarray*}
\sum_{k=0}^{m}\widetilde{b}_{nk}z_{k}=\sum_{k=0}^{m}\left[\frac{a_{nk}}{n+1}-\frac{1}{n(n+1)}\sum_{j=0}^{n-1}a_{jk}\right]z_{k} \quad \quad ~\textrm{ for all, }~   m,n\in \mathbb{N}
\end{eqnarray*}
which yields that as $m\rightarrow \infty$ that  $(Fz)_{n}=\{\Phi(Bz)\}_{n}$ for all $n\in \mathbb{N}$. Therefore, one can observe from here that $Bz\in bv(C)$ whenever $z\in Y$ if and only if $Fz\in \ell_{1}$ whenever $z\in Y$.\\
\end{proof}

\subsection{The space $bv(G)$}

In (\ref{trimtrx}), if we choose the matrix $G=G(u,v)=(g_{nk})$ instead of the matrix $A=(a_{nk})$,
then, we obtain the space $bv(G)$ as below:
\begin{eqnarray}\label{trimtrx1}
bv(G)=\left\{x=(x_{k})\in \omega: \sum_{k}\left|\sum_{j=0}^{k}u_{k}v_{j}x_{j}-\sum_{j=0}^{k-1}u_{k-1}v_{j}x_{j}\right|<\infty\right\}.
\end{eqnarray}
Using the notation (\ref{eq0}), we can denotes the space $bv(G)$ as $bv(G)=(bv)_{G(u,v)}=(\ell_{1})_{\Delta.G(u,v)}$, where $\Gamma=\gamma_{nk}=\Delta.G(u,v)$ defined by $\gamma_{nk}=(u_{n}-u_{n-1})v_{k} \quad (1\leq k<n)$; $\gamma_{nk}=u_{n}v_{n} \quad (k=n)$ and $\gamma_{nk}=0 \quad (k>n)$ for all $n$.\\

The $\Gamma-$transform of the sequence $x=(x_{k})$ defined by
\begin{eqnarray}\label{seqtrns1}
y_{k}=\sum_{j=0}^{k-1}\left(u_{k}-u_{k-1}\right)v_{j}x_{j}+u_{k}v_{k}x_{k}.\\
\end{eqnarray}

The Theorem \ref{basisweighted} and Theorem \ref{alphaweighted} can be proved as Theorem 2.8 in \cite{AB5},\cite{AB6} and Theorem 3.6 in \cite{kirisci}, respectively.\\

\begin{thm}\label{basisweighted}
Define a sequence $s^{(k)}=\{s_{n}^{(k)}\}_{n\in\mathbb{N}}$ of elements of the space $bv(G)$ for every fixed $k\in \mathbb{N}$ by
\begin{eqnarray*}
s_{n}^{(k)}= \left\{ \begin{array}{ccl}
\frac{1}{v_{n}}\left(\frac{1}{u_{k}}-\frac{1}{u_{k-1}}\right)&, & \quad (1< k <n)\\
\frac{1}{u_{n}v_{n}}&, & \quad (n=k)\\
0&, & \quad (k>n)
\end{array} \right.
\end{eqnarray*}
Therefore, the sequence $\{s^{(k)}\}_{k\in\mathbb{N}}$ is a basis for the space $bv(G)$ and any $x\in bv(G)$ has a unique representation of the form
\begin{eqnarray*}
x=\sum_{k}(\Gamma x)_{k}s^{(k)}.\\
\end{eqnarray*}
\end{thm}

\begin{thm}\label{alphaweighted}
We define the matrix $H=(h_{nk})$ as
\begin{eqnarray}\label{uvmtrx}
h_{nk}= \left\{ \begin{array}{ccl}
\frac{1}{v_{k}}\left(\frac{1}{u_{k}}-\frac{1}{u_{k-1}}\right)a_{n}&, & \quad (1\leq k < n)\\
\frac{a_{n}}{u_{n}v_{n}}&, & \quad (k=n)\\
0&, & \quad (k>n)
\end{array}\right.
\end{eqnarray}
for all $k,n\in \mathbb{N}$, where $u,v\in \mathcal{U}, a=(a_{k})\in \omega$. The $\alpha-$dual of the space $bv(G)$ is the set
\begin{eqnarray*}
d_{5}=\left\{a=(a_{k})\in \omega: \sup_{N\in \mathcal{F}}\sum_{k}\left|h_{nk}\right|<\infty \right\}.\\
\end{eqnarray*}
\end{thm}

Using the Theorem \ref{mtrxtool} and the matrix (\ref{uvmtrx}), we can give $\beta-$ and $\gamma-$duals of the space $bv(G)$ as below:\\
\begin{cor}
Let $u,v\in U$ for all $k\in \mathbb{N}$. Then,
\begin{eqnarray*}
\{bv(G)\}^{\beta}=\left\{a=(a_{k})\in \omega:\left\{\frac{1}{v_{k}}\left(\frac{1}{u_{k}}-\frac{1}{u_{k-1}}\right)a_{k}\right\}\in \ell_{1} \quad \textrm{and} \quad \left(\frac{a_{n}}{u_{n}v_{n}}\right)\in c\right\}
\end{eqnarray*}
and
\begin{eqnarray*}
\{bv(G)\}^{\gamma}=\left\{z=(z_{k})\in \omega: \left\{\frac{1}{v_{k}}\left(\frac{1}{u_{k}}-\frac{1}{u_{k-1}}\right)a_{k}\right\}\in \ell_{1} \quad \textrm{and} \quad \left(\frac{a_{n}}{u_{n}v_{n}}\right)\in \ell_{\infty}\right\}.\\
\end{eqnarray*}
\end{cor}

Following theorems can be proved as Theorem 4.2 and 4.3 in \cite{kirisci}.\\

\begin{thm}
Suppose that the entries of the infinite matrices $A=(a_{nk})$ and $B=(b_{nk})$ are connected with the relation
\begin{eqnarray*}
a_{nk}=\sum_{j=k}^{\infty} (u_{j}-u_{j-1})v_{k}b_{nj} \quad \quad ~\textrm{or }~ b_{nk}=\frac{1}{v_{k}}\left(\dfrac{1}{u_{k}}-\dfrac{1}{u_{k-1}}\right)a_{nk}
\end{eqnarray*}
for all $k,n\in\mathbb{N}$ and $Y$ be any given sequence space. Then $A\in (bv(G):Y)$ if and only if $\{a_{nk}\}_{k\in\mathbb{N}}\in \{bv(G)\}^{\beta}$ for all $n\in\mathbb{N}$ and $B\in (\ell_{1}:Y)$.\\
\end{thm}

\begin{thm}
Suppose that the entries of the infinite matrices $A=(a_{nk})$ and $T=(t_{nk})$ are connected with the relation
\begin{eqnarray*}
t_{nk}=\sum_{j=0}^{n} \left(u_n-u_{n-1}\right)v_ja_{jk}
\end{eqnarray*}
for all $k,n\in\mathbb{N}$ and $Y$ be any given sequence space. Then, $A\in (Y:bv(G))$ if and only if $T\in (Y:\ell_{1})$.\\
\end{thm}

\subsection{The space $bv(R)$}
If we choose $u_{n}=1/Q_{n}, v_{k}=q_{k}$ in the space $bv(G)$ defined by (\ref{trimtrx1}), we obtain the space $bv(R)$:
\begin{eqnarray}\label{trimtrx2}
bv(R)=\left\{x=(x_{k})\in \omega: \sum_{k}\left|\sum_{j=0}^{k}\left(q_{j}/Q_{k}\right)x_{j}-\sum_{j=0}^{k-1}\left(q_{j}/Q_{k-1}\right)x_{j}\right|<\infty\right\}.
\end{eqnarray}
we can denotes the space $bv(R)$ as $bv(R)=(bv)_{R^{t}}=(\ell_{1})_{\Delta.R^{t}}$ as the spaces $bv(G)$, where $\Sigma=\sigma_{nk}=\Delta.R^{t}$ defined by $\sigma_{nk}=(1/Q_{n}-1/Q_{n-1})q_{k} \quad (1\leq k<n)$; $\sigma_{nk}=q_{n}/Q_{n} \quad (k=n)$ and $\sigma_{nk}=0 \quad (k>n)$ for all $n$.\\

The $\Sigma-$transform of the sequence $x=(x_{k})$ defined by
\begin{eqnarray}\label{seqtrns2}
y_{k}=\sum_{j=0}^{k-1}\left(\frac{1}{Q_{k}}-\frac{1}{Q_{k-1}}\right)q_{j}x_{j}+\frac{q_{k}}{Q_{k}}x_{k}.
\end{eqnarray}

Following theorems can be proved by Theorem 2.9, Theorem 2.7 in \cite{AB3}, respectively.

\begin{thm}\label{basisRiesz}
Define a sequence $p^{(k)}=\{p_{n}^{(k)}\}_{n\in\mathbb{N}}$ of elements of the space $bv(R)$ for every fixed $k\in \mathbb{N}$ by
\begin{eqnarray*}
p_{n}^{(k)}= \left\{ \begin{array}{ccl}
\frac{\left(Q_{k}-Q_{k-1}\right)}{q_{n}}&, & \quad (1\leq k <n)\\
\frac{Q_{n}}{q_{n}}&, & \quad (n=k)\\
0&, & \quad (k>n)
\end{array} \right.
\end{eqnarray*}
Therefore, the sequence $\{p^{(k)}\}_{k\in\mathbb{N}}$ is a basis for the space $bv(R)$ and any $x\in bv(R)$ has a unique representation of the form
\begin{eqnarray*}
x=\sum_{k}(\Sigma x)_{k}p^{(k)}.\\
\end{eqnarray*}
\end{thm}

\begin{thm}\label{alphaRiesz}
We define the matrix $M=(m_{nk})$ as
\begin{eqnarray}\label{Rieszmtrx}
m_{nk}= \left\{ \begin{array}{ccl}
\frac{\left(Q_{n}-Q_{n-1}\right)}{q_{k}}a_{n}&, & \quad (1\leq k < n)\\
\frac{Q_{n}}{q_{n}}a_{n}&, & \quad (k=n)\\
0&, & \quad (k>n)
\end{array}\right.
\end{eqnarray}
for all $k,n\in \mathbb{N}$. The $\alpha-$dual of the space $bv(R)$ is the set
\begin{eqnarray*}
d_{6}=\left\{a=(a_{k})\in \omega: \sup_{N\in \mathcal{F}}\sum_{k}\left|m_{nk}\right|<\infty\right\}.\\
\end{eqnarray*}
\end{thm}

Using the Theorem \ref{mtrxtool} and the matrix (\ref{Rieszmtrx}), we can give $\beta-$ and $\gamma-$duals of the space $bv(R)$ as below:\\
\begin{cor}
\begin{eqnarray*}
\{bv(R)\}^{\beta}=\left\{a=(a_{k})\in \omega:\left\{\frac{\left(Q_{n}-Q_{n-1}\right)}{q_{k}}a_{n}\right\}\in \ell_{1} \quad \textrm{and} \quad \left(\frac{Q_{n}}{q_{n}}a_{n}\right)\in c\right\}
\end{eqnarray*}
and
\begin{eqnarray*}
\{bv(R)\}^{\gamma}=\left\{z=(z_{k})\in \omega: \left\{\frac{\left(Q_{n}-Q_{n-1}\right)}{q_{k}}a_{n}\right\}\in \ell_{1} \quad \textrm{and} \quad \left(\frac{Q_{n}}{q_{n}}a_{n}\right)\in \ell_{\infty}\right\}.\\
\end{eqnarray*}
\end{cor}

\begin{thm}
Suppose that the entries of the infinite matrices $A=(a_{nk})$ and $B=(b_{nk})$ are connected with the relation
\begin{eqnarray*}
a_{nk}=\sum_{j=k}^{\infty}\left(\frac{1}{Q_{n}}-\frac{1}{Q_{n-1}}\right)q_{j}b_{nj}
\quad \quad ~\textrm{or}~ b_{nk}=\frac{\left(Q_{n}-Q_{n-1}\right)}{q_{k}}a_{nk}
\end{eqnarray*}
for all $k,n\in\mathbb{N}$ and $Y$ be any given sequence space. Then $A\in (bv(R):Y)$ if and only if $\{a_{nk}\}_{k\in\mathbb{N}}\in \{bv(R)\}^{\beta}$ for all $n\in\mathbb{N}$ and $B\in (\ell_{1}:Y)$.\\
\end{thm}

\begin{thm}
Suppose that the entries of the infinite matrices $A=(a_{nk})$ and $W=(w_{nk})$ are connected with the relation \begin{eqnarray*}
w_{nk}=\sum_{k=0}^{n}\left(\frac{1}{Q_{n}}-\frac{1}{Q_{n-1}}\right)q_{j}a_{jk}
\end{eqnarray*}
for all $k,n\in\mathbb{N}$ and $Y$ be any given sequence space. Then, $A\in (Y:bv(R))$ if and only if $W\in (Y:\ell_{1})$.\\
\end{thm}

\section{Conclusion}
The matrix domain $X_{\Delta}$ for $X=\{\ell_{\infty}, c, c_{0}\}$ is called the \emph{difference sequence spaces},
which was firstly defined and studied by K{\i}zmaz \cite{Kizmaz}. If we choose $X=\ell_{1}$, the space $\ell_{1}({\Delta})$
is called \emph{the space of all sequences of bounded variation} and denote by $bv$. The space $bv_{p}$ consisting of all sequences
whose differences are in the space $\ell_{p}$. The space $bv_{p}$ was introduced by Başar and Altay \cite{AB2}. More recently, the sequence spaces $bv$
are studied in \cite{AB2}, \cite{BAM}, \cite{Bennet}, \cite{imamiri}, \cite{JarMal}, \cite{kirisci3}, \cite{MalRakZiv}.

Several authors studied deal with the sequence spaces which obtained with the domain of the triangle matrices. It can be seen that the matrix domains for the matrix $C_{1}$ which known as Ces\`{a}ro mean of order one in \cite{basar}, \cite{NgLee}, \cite{SengBas}; for the generalized weighted mean in \cite{AB5}, \cite{AB6}, \cite{BasKa3}, \cite{BasKa4} , \cite{kirisci2}, \cite{kirisci}, \cite{MalOz}, \cite{MalSav}, \cite{PKS}; for the Riesz mean in \cite{AB3}, \cite{AB4}, \cite{Bas1}, \cite{BasKa0}, \cite{BasKa2}, \cite{BasKay}, \cite{BasOz}, \cite{BasOz2}, \cite{KonBas}, \cite{Mal1}. Further, different works related to the matrix domain of the sequence spaces can be seen in \cite{Basarkitap}.

In this work, we give well-known results related to some properties, dual spaces and matrix transformations of the sequence space $bv$ and introduce the matrix domain of space $bv$ with arbitrary triangle matrix $A$. Afterward, we choose the matrix $A$ as Ces\`{a}ro mean of order one, generalized weighted mean and Riesz mean and compute $\alpha-, \beta-, \gamma-$duals of those spaces. And also, we characterize the matrix classes of the spaces $bv(C)$, $bv(G)$, $bv(R)$.

As a natural continuation of this paper, one can study the domain of different matrices instead of $A$. Additionally, sequence spaces in this paper can be defined by a index $p$ for $1\leq p <\infty$ and a  bounded sequence of strictly positive real numbers $(p_{k})$ for $0< p_{k} \leq 1$ and $1< p_{k} <\infty$ and the concept almost convergence. And also it may be characterized several classes of matrix transformations between new sequence spaces in this work and sequence spaces which obtained with the domain of different matrices.

\end{document}